\documentclass[oneside,reqno,10pt]{amsart}
\usepackage{amsfonts,amssymb,amsmath,amsthm}

\usepackage{hyperref}

\theoremstyle{plain} 
\newtheorem{theorem}             {Theorem} 
\newtheorem{corollary}  [theorem]{Corollary}
\newtheorem{algorithm}  [theorem]{Algorithm}

\theoremstyle{remark}

\DeclareMathOperator{\ind}{ind}
\def\modd#1 #2{#1\ \mbox{\rm (mod}\ #2\mbox{\rm )}}

\begin{document}

\author{Daniel Tsai}
\address{Nagoya University, Graduate School of Mathematics, 464-8602, Furocho, Chikusa-ku, Nagoya, Japan}
\email{dsai@outlook.jp}
\subjclass[2010]{Primary 11A63; Secondary 11A07, 11A25}

\title{Repeated concatenations in residue classes}

\begin{abstract}
We give an algorithm to determine all the repeated concatenations, in a given base, of a natural number in a residue class. The author recently describes a particular sequence of $v$-palindromes that inspires this investigation. We also generalize this sequence and discuss how there could be variations of the investigation we do in this paper.
\end{abstract}
\maketitle

\section{Introduction}
The author \cite{tsai,tsai18} defined the notion of a natural number being a {\em $v$-palindrome}, which we now define. Let $n\geq1$ be an integer. The number formed by writing the decimal digits of $n$ in reverse order is denoted by $r(n)$. The additive arithmetical function $v\colon \mathbb{N}\to \mathbb{Z}$ is characterized by having, for a prime power $p^\alpha$, $v(p^\alpha)=p$ if $\alpha=1$ and $v(p^\alpha)=p+\alpha$ if $\alpha\geq2$. Then, $n$ is a $v$-{\it palindrome} if $10\nmid n$, $n\ne r(n)$, and $v(n)=v(r(n))$.

The sequence of $v$-palindromes is A338039 in the {\it On-line Encyclopedia of Integer Sequences} (OEIS)~\cite{oeis}. In particular, all the repeated concatenations of $18$,
\begin{equation}\label{ten}
  18,1818,181818,\ldots,
\end{equation}
are $v$-palindromes. Other instances where all the repeated concatenations of a number are $v$-palindromes include
\begin{align}
  \label{lexus} 198,198198,198198198,\ldots, \\
  \label{new} 576,576576,576576576,\ldots.
\end{align}
Even if not all the repeated concatenations of a number are $v$-palindromes, some might be, and the pattern is periodic \cite[Theorem 1]{tsai}. The author \cite{tsai21} gave a method to determine whether a repeated concatenation of a number is a $v$-palindrome. For instance, a repeated concatenation of $117$ is $v$-palindromic if and only if the number of copies of $117$ is a multiple of $2054$ \cite[Table 4]{tsai21}. Using this method, lots of $v$-palindromes that are repeated concatenations can potentially be found. Indeed, the sequence of such $v$-palindromes is A338166 in OEIS \cite{oeis}.

Harminc and Sot\'ak \cite{harminc} showed that an arithmetic sequence $(a+bm)_{m\geq 0}$, where $a,b\geq 1$ are integers, contains a palindrome if and only if it is not the case that $a\equiv b\equiv \modd{0} {10}$, in which case $(a+bm)_{m\geq 0}$ will contain infinitely many palindromes. An arithmetic sequence cannot consist entirely of palindromes because, in fact, Pongsriiam \cite{pongsriiam} showed that the longest arithmetic progression of palindromes has length $10$. Analogously, we can consider whether an arithmetic sequence contains a $v$-palindrome. Since all the numbers \eqref{ten} are $v$-palindromes, one way to look for $v$-palindromes in an arithmetic sequence would be to look specifically for a repeated concatenation of $18$. The same can be said for \eqref{lexus} and \eqref{new}.

In this paper, instead of looking for $v$-palindromes (the structure of whose totality is still very mysterious) in an arithmetic sequence, we look for repeated concatenations of a number. Although this investigation is inspired by considering a result on palindromes \cite{harminc} analogously for $v$-palindromes, it is not directly about $v$-palindromes. Whether it is looking for palindromes in an arithmetic sequence \cite{harminc}, $v$-palindromes in an arithmetic sequence, or repeated concatenations of a number in an arithmetic sequence, these are all special instances of looking for terms in a sequence with a certain property. There are also investigations that find the density of the terms in a sequence with a certain property and involve palindromes \cite{banks,cilleruelo,cilleruelo13,luca}. 

We also generalize \eqref{ten} into Theorem \ref{general}, which relates $v$-palindromes and palindromes, and discuss in Section \ref{further} how Theorem \ref{general} could inspire variations of the investigation done in this paper. In Section \ref{the}, we make more precise what is to be investigated, namely, determining all the repeated concatenations of a number in a residue class; an arithmetic sequence is essentially a residue class, and we think in terms of residue classes rather than arithmetic sequences. In Section \ref{nota}, we fix some notation and make some conventions.

\subsection{The problem}\label{the}
We state our problem more precisely. Let $n\geq 1$ be an integer with base $b\geq 2$ representation $(a_{L-1}a_{L-2}\cdots a_0)_b$, where $0\leq a_{L-1},a_{L-2},\ldots,a_0<b$ are integers and $a_{L-1}\ne0$. Let $m\geq 1$ and $a$ be integers. Let the $k\geq 1$ times repeated concatenation in base $b$ of $n$ be denoted by $n(k)_b$. That is,
\begin{align}
  n(k)_b & =(\underbrace{a_{L-1}a_{L-2}\cdots a_0a_{L-1}a_{L-2}\cdots a_0\cdots\cdots a_{L-1}a_{L-2}\cdots a_0}_\text{$k$ copies of $a_{L-1}a_{L-2}\cdots a_0$})_b \nonumber \\
  & =n(1+b^L+\cdots + b^{(k-1)L})=n\cdot\frac{1-b^{Lk}}{1-b^L}. \label{wd}
\end{align}
For instance, $18(3)_{10}=181818$ but
\begin{equation*}
  18(3)_2=(10010)_2(3)_2=(100101001010010)_2=19026.
\end{equation*}
Our problem is to solve the congruence
\begin{equation}\label{tosolve}
  n(k)_b\equiv a\pmod{m}
\end{equation}
for integers $k\geq 1$. That is, to find all the repeated concatenations in base $b$ of $n$ in the residue class $a+m\mathbb{Z}$.
We give an algorithm (Algorithm \ref{crux}) in Section \ref{sol1}, to determine all $k\geq 1$ satisfying \eqref{tosolve}, when $m$ is a prime power. Then, we give Algorithm \ref{gen} in Section \ref{sol2} for general $m$, which simply consists of multiple applications of Algorithm \ref{crux}.  The set of all $k\geq 1$ satisfying \eqref{tosolve} will be denoted by $K$.

In Section \ref{conc}, we give a concrete example using Algorithm \ref{gen}.  In Section \ref{further}, we discuss how Theorem \ref{general} could suggest variations of the congruence \eqref{tosolve}.

\subsection{Notation and conventions}\label{nota}
We fix the following notation to be used.
\begin{itemize}
\item For integers $c\ne0$ and $\gamma\geq 0$ and a prime $p$, the notation $p^\gamma \parallel c$ means that $p^\gamma\mid c$ but $p^{\gamma+1}\nmid c$.
\item The Iverson symbol $[\cdot]$ is defined for a statement $P$ by $[P]=1$ or $[P]=0$ according as whether $P$ is true or false.
\item In a congruence relation modulo $m$, a notation $x^{-1}$ denotes an inverse of $x$ modulo $m$.
\item If $g$ is a primitive root modulo $m$ and $\gcd(x,m)=1$, then $\ind_{g,m} x$ denotes the index of $x$ to the base $g$ modulo $m$.
\end{itemize}   
We also make the following conventions for our algorithms.
\begin{itemize}
  \item Once an output is reached, the algorithm terminates.
  \item An output written as a condition on $k$ means that we output the set of all integers $k\geq 1$ satisfying that condition.
\end{itemize}

\section{When \texorpdfstring{$m=p^\alpha$}{m=pa} is a prime power}
\label{sol1}
We first consider the case when $m=p^\alpha$ is a prime power.
By \eqref{wd}, \eqref{tosolve} is equivalent to
\begin{equation}\label{con1}
  n\cdot \frac{1-b^{Lk}}{1-b^L} \equiv a\pmod{p^\alpha}.
\end{equation}
Put $d=\gcd(n,p^\alpha)$. If $d\nmid a$, then there is no solution for $k$, i.e., $K=\varnothing$. Thus assume that $d\mid a$. Then \eqref{con1} is equivalent to
\begin{equation*}
  \frac{n}{d}\cdot \frac{1-b^{Lk}}{1-b^L} \equiv \frac{a}{d}\pmod{\frac{p^\alpha}{d}},
\end{equation*}
which is equivalent to
\begin{equation}\label{above}
  \frac{1-b^{Lk}}{1-b^L} \equiv \frac{a}{d}\cdot \left(\frac{n}{d}\right)^{-1}\pmod{\frac{p^\alpha}{d}}.
\end{equation}
Put $p^\alpha/d=p^{\alpha_1}$ and let $a_1\equiv a/d\cdot \modd{(n/d)^{-1}} {p^{\alpha_1}}$. Then \eqref{above} is equivalent to
\begin{equation}\label{above2}
   \frac{1-b^{Lk}}{1-b^L} \equiv a_1\pmod{p^{\alpha_1}}.
\end{equation}
Suppose that $p^\beta\parallel 1-b^L$, then \eqref{above2} is equivalent to
\begin{equation*}
  \frac{1-b^{Lk}}{p^\beta} \equiv a_1\cdot \frac{1-b^L}{p^\beta}\pmod{p^{\alpha_1}},
\end{equation*}
or equivalently,
\begin{equation*}
1-b^{Lk} \equiv a_1 (1-b^L)\pmod{p^{\alpha_1+\beta}},
\end{equation*}
or equivalently,
\begin{equation}\label{above3}
  b^{Lk}\equiv 1-a_1(1-b^L)\pmod{p^{\alpha_1+\beta}}.
\end{equation}
Put $p^{\alpha_1+\beta}=p^{\alpha_2}$ and let $a_2\equiv \modd{1-a_1(1-b^L)} {p^{\alpha_2}}$. Then \eqref{above3} is equivalent to
\begin{equation}\label{last}
  b^{Lk}\equiv a_2\pmod{p^{\alpha_2}}.
\end{equation}
If $\alpha_2 = 0$, then $K=\mathbb{N}$. Thus assume that $\alpha_2\geq 1$. There will be two cases, according to as whether there is not or is a primitive root modulo $p^{\alpha_2}$, and we consider them in Sections \ref{no} and \ref{is}, respectively. Recall that there is no primitive root modulo $p^{\alpha_2}$ if and only if $p=2$ and $\alpha_2\geq 3$.

\subsection{In case \texorpdfstring{$p=2$}{p=2} and \texorpdfstring{$\alpha_2\geq 3$}{a2>=3}}\label{no}
In case $p=2$ and $\alpha_2\geq 3$, the congruence \eqref{last} is equivalent to
\begin{equation}\label{bec1}
  b^{Lk}\equiv a_2\pmod{2^{\alpha_2}}.
\end{equation}
If $b\not\equiv \modd{a_2} {2}$, then $K=\varnothing$. Thus assume that $b\equiv \modd{a_2} {2}$. We consider the cases $b\equiv a_2\equiv \modd{0} {2}$ and $b\equiv a_2\equiv \modd{1} {2}$ in the next two paragraphs, respectively.

In case $b\equiv a_2\equiv\modd{0} {2}$, write $b=2^\delta b_1$, where $2^\delta \parallel b$. If $a_2\equiv\modd{0} {2^{\alpha_2}}$, then $K = \{k\in\mathbb{N}\colon k\geq \alpha_2/(\delta L)\}$. Thus assume that $a_2\not\equiv\modd{0} {2^{\alpha_2}}$. Write $a_2=2^\varepsilon a_3$, where  $2^\varepsilon \parallel a_2$. Then \eqref{bec1} is equivalent to
\begin{equation}\label{mid}
  2^{\delta L k}b^{Lk}_1\equiv 2^\varepsilon a_3\pmod{2^{\alpha_2}}.
\end{equation}
Since $a_2\not\equiv\modd{0} {2^{\alpha_2}}$, $\varepsilon<\alpha_2$,
therefore \eqref{mid} implies that
\begin{equation*}
  2^{\delta L k}b^{Lk}_1\equiv 0\pmod{2^{\varepsilon}}.
\end{equation*}
Hence, we need to have $\delta L k\geq \varepsilon$. Now assume that $\delta L k\geq  \varepsilon$. Then \eqref{mid} holds if and only if
\begin{equation}\label{sub}
  2^{\delta L k-\varepsilon}b^{Lk}_1\equiv a_3\pmod{2^{\alpha_2-\varepsilon}}.
\end{equation}
If $\delta L k>\varepsilon$, then the above congruence cannot hold because the two sides are of opposite parity. Hence, we need to have $k=\varepsilon / (\delta L)$. If $\varepsilon/(\delta L)$ is not an integer, then $K=\varnothing$. Thus assume that $\varepsilon/(\delta L)$ is an integer. Letting $k=\varepsilon/(\delta L)$, \eqref{sub} becomes
\begin{equation}\label{ho}
  b^{\varepsilon/\delta}_1\equiv a_3\pmod{2^{\alpha_2-\varepsilon}}.
\end{equation}
If \eqref{ho} holds, then $K=\{\varepsilon/(\delta L)\}$, otherwise $K=\varnothing$.

In case $b\equiv a_2\equiv\modd{1} {2}$, by the structure of $(\mathbb{Z}/2^{\alpha_2}\mathbb{Z})^\times$, there exist unique integers $0\leq\mu_1,\mu_2<2$ and $0\leq \nu_1,\nu_2<2^{\alpha_2-2}$ such that
$b\equiv \modd{(-1)^{\mu_1}5^{\nu_1}} {2^{\alpha_2}}$ and $a_2\equiv \modd{(-1)^{\mu_2}5^{\nu_2}} {2^{\alpha_2}}$. Hence, \eqref{bec1} is equivalent to
\begin{equation*}
  (-1)^{\mu_1 Lk}5^{\nu_1 Lk}\equiv (-1)^{\mu_2}5^{\nu_2}\pmod{2^{\alpha_2}},
\end{equation*}
which holds if and only if both of the congruences
\begin{align}
  \label{first}\mu_1 Lk & \equiv \mu_2\pmod{2}, \\
  \label{second} \nu_1 Lk &\equiv \nu_2\pmod{2^{\alpha_2-2}}
\end{align}
hold.
We solve this system of congruences for $k$. If $\mu_1 L$ is even and $\mu_2$ odd, then \eqref{first} cannot hold, thus $K=\varnothing$. Thus assume that $K\neq \varnothing$. We divide into two cases as follows.
\begin{itemize}
  \item[{\rm (i)}] \label{step1} {\bf If $\mu_1 L$ is odd}: \eqref{first} is equivalent to $k\equiv \modd{\mu_2} {2}$. We solve \eqref{second} in the usual way. Put $f=\gcd(\nu_1 L,2^{\alpha_2-2})$.  If $f\nmid \nu_2$, then \eqref{second} cannot hold, thus $K=\varnothing$. Thus assume that $f\mid \nu_2$. Then \eqref{second} is equivalent to
  \begin{equation}\label{implies}
    k\equiv \frac{\nu_2}{f}\left(\frac{\nu_1L}{f}\right)^{-1}\pmod{\frac{2^{\alpha_2-2}}{f}}.
  \end{equation}
  If $2^{\alpha_2-2}/f=1$, then \eqref{implies} always hold, and so
  \begin{equation}
    K=\{k\in\mathbb{N}\colon k\equiv\modd{\mu_2} {2}\}.
  \end{equation}
  Thus assume that $2^{\alpha_2-2}/f>1$. Then \eqref{implies} implies that
  \begin{equation*}
    k\equiv \frac{\nu_2}{f}\left(\frac{\nu_1L}{f}\right)^{-1} \equiv \frac{\nu_2}{f}\pmod{2}.
  \end{equation*}
  If
  \begin{equation}\label{if}
    \mu_2\equiv  \frac{\nu_2}{f}\pmod{2},
  \end{equation}
  then
  \begin{equation*}
  K=\left\{k\in\mathbb{N}\colon k\equiv \frac{\nu_2}{f}\left(\frac{\nu_1L}{f}\right)^{-1}\pmod{\frac{2^{\alpha_2-2}}{f}}\right\}.
  \end{equation*}
  If \eqref{if} does not hold, then $K=\varnothing$.
  \item[{\rm (ii)}] {\bf If $\mu_1L$ and $\mu_2$ are both even}: \eqref{first} always hold, so we are left with solving just \eqref{second}, which we do as in the second to sixth sentences in case {\rm (i)}.
\end{itemize}

\subsection{In case \texorpdfstring{$p$}{p} is odd or \texorpdfstring{$\alpha_2<3$}{a2<3}}\label{is}
We now consider the case when $p$ is odd or $\alpha_2<3$. The congruence \eqref{last} implies that $b^{Lk}\equiv \modd{a_2} {p}$. Consequently, if $[p\mid b]\ne[p\mid a_2]$, then $K=\varnothing$. Thus assume that $[p\mid b]=[p\mid a_2]$. In case $[p\mid b]=[p\mid a_2]=1$, we solve \eqref{last} in the same way as in the case when $p=2$, $\alpha_2\geq 3$, and $b\equiv a_2\equiv\modd{0} {2}$, described in the second paragraph of Section \ref{no}. Thus assume that $[p\mid b]=[p\mid a_2]=0$.

Let $g$ be a primitive root modulo $p^{\alpha_2}$. Then \eqref{last} is equivalent to
\begin{equation}\label{just}
  Lk\ind_{g,p^{\alpha_2}} b\equiv \ind_{g,p^{\alpha_2}} a_2\pmod{p^{\alpha_2-1}(p-1)}.
\end{equation}
So we just have to solve \eqref{just}, which we do in the usual way. Put
\begin{equation*}
  f=\gcd(L\ind_{g,p^{\alpha_2}} b,p^{\alpha_2-1}(p-1)).
\end{equation*}
If $f\nmid \ind_{g,p^{\alpha_2}} a_2$, then \eqref{just} cannot hold, thus $K=\varnothing$. Thus assume that $f\mid \ind_{g,p^{\alpha_2}} a_2$. Then \eqref{just} is equivalent to
  \begin{equation*}
    k\equiv \frac{\ind_{g,p^{\alpha_2}} a_2}{f}\left(\frac{L\ind_{g,p^{\alpha_2}} b}{f}\right)^{-1}\pmod{\frac{p^{\alpha_2-1}(p-1)}{f}},
  \end{equation*}
  and so
  \begin{equation*}
    K=\left\{k\in\mathbb{N}\colon k\equiv \frac{\ind_{g,p^{\alpha_2}} a_2}{f}\left(\frac{L\ind_{g,p^{\alpha_2}} b}{f}\right)^{-1}\pmod{\frac{p^{\alpha_2-1}(p-1)}{f}}\right\}.
  \end{equation*}

\subsection{Algorithm when \texorpdfstring{$m=p^\alpha$}{m=pa} is a prime power}\label{alg1}
Up to this point in Section \ref{sol1}, we have shown how to determine all $k\geq 1$ satisfying \eqref{tosolve}, when $m$ is a prime power. We now summarize the process into the following algorithm. 
\begin{algorithm}\label{crux}
  Given integers $n\geq 1$, $b\geq 2$, $a\in\mathbb{Z}$, and a prime power $m=p^\alpha$, this algorithm computes the set $K$ of integers $k\geq 1$ satisfying \eqref{tosolve}.
  \begin{itemize}
    \item[{\rm (I)}] Put $d=\gcd(n,p^\alpha)$. If $d\nmid a$, output $K=\varnothing$.
    \item[{\rm (II)}] Let the number of base $b$ digits of $n$ be denoted by $L$. Put $p^\alpha/d=p^{\alpha_1}$ and suppose that $p^\beta\parallel 1-b^L$. Put $\alpha_2=\alpha_1+\beta$. If $\alpha_2=0$, output $K=\mathbb{N}$. Let $a_1,a_2\in\mathbb{Z}$ be such that
    \begin{align*}
      a_1&\equiv \frac{a}{d}\cdot\left(\frac{n}{d}\right)^{-1}\pmod{p^{\alpha_1}}, \\
      a_2&\equiv 1-a_1(1-b^L)\pmod{p^{\alpha_2}}.
    \end{align*}
    If $p$ is odd or $\alpha_2<3$, go to step {\rm (XII)}.
    \item[{\rm (III)}] If $b\not\equiv \modd{a_2} {2}$, output $K=\varnothing$. If $b\equiv a_2\equiv \modd{1} {2}$, go to step {\rm (VII)}.
    \item[{\rm (IV)}] Suppose that $p^\delta\parallel b$. If $a_2\equiv \modd{0} {p^{\alpha_2}}$, output $k\geq \alpha_2/(\delta L)$.
    \item[{\rm (V)}] Suppose that $p^\varepsilon\parallel a_2$. If $\delta L\nmid \varepsilon$, output $K=\varnothing$.
    \item[{\rm (VI)}] If $b^{\varepsilon/\delta}\equiv \modd{a_2} {p^{\alpha_2}}$, output $k=\varepsilon/(\delta L)$. Output $K=\varnothing$.
    \item[{\rm (VII)}] Let $0\leq \mu_1,\mu_2<2$ and $0\leq \nu_1,\nu_2<2^{\alpha_2-2}$ be integers such that
    \begin{align*}
      b&\equiv(-1)^{\mu_1}5^{\nu_1}\pmod{2^{\alpha_2}}, \\
      a_2&\equiv (-1)^{\mu_2}5^{\nu_2}\pmod{2^{\alpha_2}}.
    \end{align*}
    If $2\mid\mu_1 L$ and $2\nmid \mu_2$, output $K=\varnothing$.
    \item[{\rm (VIII)}] Put $f=\gcd(\nu_1 L, 2^{\alpha_2-2})$. If $f\nmid \nu_2$, output $K=\varnothing$. If $2\nmid\mu_1 L$, go to step {\rm (X)}.
    \item[{\rm (IX)}] Output 
    \begin{equation*}
      k\equiv \frac{\nu_2}{f}\left(\frac{\nu_1L}{f}\right)^{-1} \pmod{ \frac{2^{\alpha_2-2}}{f}}.
    \end{equation*}
    \item[{\rm (X)}] If $f=2^{\alpha_2-2}$, output $k\equiv \modd{\mu_2} {2}$.     
    \item[{\rm (XI)}] If $\mu_2\not\equiv\modd{\frac{\nu_2}{f}} {2}$, output $K=\varnothing$. Go to step {\rm (IX)}.
    \item[{\rm (XII)}] If $[p\mid b]\neq[p\mid a_2]$, output $K=\varnothing$. If $[p\mid b]=[p\mid a_2]=1$, go to step {\rm (IV)}.
    \item[{\rm (XIII)}] Let $g$ be a primitive root modulo $p^{\alpha_2}$ and put $f=\gcd(L\ind_{g,p^{\alpha_2}} b, p^{\alpha_2-1}(p-1))$. If $f\nmid \ind_{g,p^{\alpha_2}} a_2$, output $K=\varnothing$.
    \item[{\rm (XIV)}] Output
    \begin{equation*}
      k\equiv \frac{\ind_{g,p^{\alpha_2}} a_2}{f}\left(\frac{L \ind_{g,p^{\alpha_2}} b}{f}\right)^{-1}\pmod{\frac{p^{\alpha_2-1}(p-1)}{f}}.
    \end{equation*}
  \end{itemize}
\end{algorithm}

\section{For general modulus \texorpdfstring{$m$}{m}}
\label{sol2}
The case when $m=p^\alpha$ is a prime power is treated in Section \ref{sol1}. We now solve the congruence \eqref{tosolve} for $k$, for a general modulus $m$. When $m=1$, clearly $K=\mathbb{N}$. Thus assume that $m>1$. Let the canonical factorization of $m$ be $m=p^{\alpha_1}_1\cdots p^{\alpha_r}_r$. Then the congruence \eqref{tosolve} is the conjunction of
\begin{equation}\label{sat}
  n(k)_b\equiv a\pmod{p^{\alpha_j}_j},
\end{equation}
for $1\leq j\leq r$.
For each $1\leq j\leq r$, we can solve the above congruence for $k$ by the process of Section \ref{sol1}, i.e., Algorithm \ref{crux}, obtaining a solution set $K_j$. Consequently, $K=K_1\cap\cdots\cap K_r$. In actually finding $K$, we can use the Chinese remainder theorem. We summarize this into the following algorithm.
\begin{algorithm}\label{gen}
  Given integers $n\geq 1$, $b\geq 2$, $a\in\mathbb{Z}$, and $m\geq 1$, this algorithm computes the set $K$ of integers $k\geq 1$ satisfying \eqref{tosolve}.
  \begin{itemize}
    \item[{\rm (I)}] If $m=1$, output $K=\mathbb{N}$.
    \item[{\rm (II)}] Let the canonical factorization of $m$ be $m=p^{\alpha_1}_1\cdots p^{\alpha_r}_r$. For each $1\leq j\leq r$, compute the set $K_j$ of integers $k\geq 1$ satisfying \eqref{sat} by using Algorithm \ref{crux}. Output $K=K_1\cap\cdots\cap K_r$.
  \end{itemize}
\end{algorithm}

\section{A concrete example}
\label{conc}
In this section, we give a concrete example using Algorithm \ref{gen}. Consider the congruence
\begin{equation}\label{exn}
  18(k)_3\equiv 2 \pmod{208}.
\end{equation}
We find the set $K$ of integers $k\geq 1$ satisfying the above congruence by using Algorithm \ref{gen} with $n=18$, $b=3$, $a=2$, and $m=208$. Since $m>1$, we go to step (II). We have the canonical factorization $208=2^4\cdot 13$. In Sections \ref{K1} and \ref{K2}, by using Algorithm \ref{crux}, we find the sets $K_1$ and $K_2$ of integers $k\geq 1$ satisfying the congruences
\begin{equation*}
  18(k)_3\equiv 2\pmod{2^4}\quad\text{and} \quad 18(k)_3\equiv 2\pmod{13},
\end{equation*}
respectively. Then, in Section \ref{K}, we consider $K=K_1\cap K_2$.

\subsection{Computation of \texorpdfstring{$K_1$}{K1}}\label{K1}

We use Algorithm \ref{crux} with $n=18$, $b=3$, $a=2$, and $m=2^4$.
\begin{itemize}
  \item[{\rm (I)}] Put $d=\gcd(18,2^4)=2$. Since $d=2\mid 2=a$, we go to step (II).
  \item[{\rm (II)}] Since $18 = 200_3$, $L=3$. Since $2^4/2=2^3$, $\alpha_1=3$. Since $1-b^L=1-3^3=-26$, $\beta=1$. Put $\alpha_2=\alpha_1+\beta=3+1=4\ne 0$. Since
  \begin{align*}
    \frac{2}{2}\cdot\left(\frac{18}{2}\right)^{-1}&=9^{-1}\equiv1^{-1}\equiv1\pmod{2^3}, \\
    1-1\cdot(-26)&=1+26 \equiv-5\pmod{2^4},
  \end{align*}   
  we can choose $a_1=1$ and $a_2=-5$. Since $p=2$ and $\alpha_2=4\geq 3$, we go to step (III).
  \item[{\rm (III)}] Since $3\equiv -5\equiv \modd{1} {2}$, we go to step (VII).
  \item[{\rm (VII)}] Since
  \begin{align*}
    b&=3\equiv (-1)^1\cdot 5^3\pmod{2^4}, \\
    a_2&=-5\equiv (-1)^1\cdot 5^1\pmod{2^4},
  \end{align*}
  $\mu_1=\mu_2=1$, $\nu_1=3$, and $\nu_2=1$. Since $2\nmid3=\mu_1 L$, we go to step (VIII).
  \item[{\rm (VIII)}] Put $f=\gcd(\nu_1L, 2^{\alpha_2-2})=\gcd(9,2^{2})=1$. Then $f=1\mid \nu_2$. Since $2\nmid 3=\mu_1L$, we go to step (X).
  \item[{\rm (X)}] Since $f=1<2^2=2^{\alpha_2-2}$, we go to step (XI).
  \item[{\rm (XI)}] Since $\mu_2=\nu_2/f$,
  we go to step (IX).
  \item[{\rm (IX)}] Since
  \begin{equation*}
    \frac{1}{1}\cdot\left(\frac{9}{1}\right)^{-1}\equiv 1\pmod{2^2},
  \end{equation*}
  we obtain that $k\equiv \modd{1} {4}$.
\end{itemize}
Therefore we have computed that
\begin{equation*}
  K_1=\{k\in\mathbb{N}\colon k\equiv 1\pmod{4}\}.
\end{equation*}

\subsection{Computation of \texorpdfstring{$K_2$}{K2}}\label{K2}

We use Algorithm \ref{crux} with $n=18$, $b=3$, $a=2$, and $m=13$.
\begin{itemize}
  \item[{\rm (I)}] Put $d=\gcd(18,13)=1$. Since $d=1\mid 2=a$, we go to step (II).
  \item[{\rm (II)}] Since $18 = 200_3$, $L=3$. Since $13/1=13^1$, $\alpha_1=1$. Since $1-b^L=-26$, $\beta=1$. Put $\alpha_2=\alpha_1+\beta=1+1=2\ne0$. Since
  \begin{align*}
    \frac{2}{1}\cdot\left(\frac{18}{1}\right)^{-1}&\equiv 2\cdot 5^{-1}\equiv3 \pmod{13}, \\
    1-3(-26)&=1+3\cdot 26\equiv 79\pmod{13^2},
  \end{align*}
  we can choose $a_1=3$ and $a_2=79$. Since $p=13$, we go to step (XII).
  \item[{\rm (XII)}] Since $[13\mid 3]=[13\mid 79]=0$, we go to step (XIII).
  \item[{\rm (XIII)}] A primitive root modulo $13^2$ is $g=2$. We have $\ind_{2,13^2} 3=124$ and $\ind_{2,13^2} 79 = 24$. Put $f=\gcd(3\cdot 124,13\cdot 12)=12$. Since $f=12\mid 24=\ind_{2,13^2} 79$, we go to step (XIV).
  \item[{\rm (XIV)}] Since
  \begin{equation*}
    \frac{24}{12}\cdot\left(\frac{3\cdot124}{12}\right)^{-1}=2\cdot 31^{-1}\equiv  2\cdot 5^{-1}\equiv3  \pmod{13},
  \end{equation*}
  we obtain that $k\equiv \modd{3} {13}$.
\end{itemize}
Therefore we have computed that
\begin{equation*}
  K_2=\{k\in\mathbb{N}\colon k\equiv 3\pmod{13}\}.
\end{equation*}

\subsection{Computation of \texorpdfstring{$K$}{K}}\label{K}

In Sections \ref{K1} and \ref{K2}, we computed respectively that $K_1=\{k\in\mathbb{N}\colon k\equiv \modd{1} {4}\}$ and $K_2=\{k\in\mathbb{N}\colon k\equiv \modd{3} {13}\}$. By the Chinese remainder theorem, $K=K_1\cap K_2=\{k\in\mathbb{N}\colon k\equiv \modd{29} {52}\}$.

Therefore we showed that, for $k\geq1$, \eqref{exn} holds if and only if $k\equiv \modd{29} {52}$. In other words, because $18 = 200_3$, what we showed is that if we repeatedly concatenate $k$ times the digits $200$ and consider the resulting number $R$ in base $3$, then $R\equiv \modd{2} {208}$ if and only if $k\equiv \modd{29} {52}$.

\section{A generalization of \texorpdfstring{\eqref{ten}}{ten}}
\label{sec-5}
  We said in the Introduction that all the repeated concatenations \eqref{ten} of $18$ are $v$-palindromes. In this section, we prove a generalization, Theorem \ref{general}, of \eqref{ten}. Then, we deduce two corollaries, including \eqref{ten}.
  \begin{theorem}\label{general}
  If $\rho$ is a palindrome in base $10$ whose digits consist entirely of $0$'s and $1$'s, then $18\rho$ is a $v$-palindrome.
\end{theorem}
\begin{proof}
  When read from left to right, $\rho$ must be formed by $a_1$ ones, followed by $a_2$ zeros, followed by $a_3$ ones, and so on until lastly, $a_{2r-1}$ ones, where $r,a_1,a_2,\ldots,a_{2r-1}$ are positive integers such that $a_i=a_{2r-i}$ for $1\leq i\leq 2r-1$. Writing $\rho$ out,
  \begin{equation*}
    \rho=\underbrace{1\cdots 1}_\text{$a_1$ ones}\overbrace{0\cdots 0}^\text{$a_2$ zeros}\underbrace{1\cdots 1}_\text{$a_3$ ones}\cdots\cdots\underbrace{1\cdots 1}_\text{$a_3$ ones}\overbrace{0\cdots 0}^\text{$a_2$ zeros}\underbrace{1\cdots 1}_\text{$a_1$ ones}.
  \end{equation*}
  Thus
  \begin{align*}
    18\rho&=1\underbrace{9\cdots9}_\text{$a_1-1$}8\overbrace{0\cdots0}^\text{$a_2-1$}1\underbrace{9\cdots9}_\text{$a_3-1$}8\cdots\cdots1\underbrace{9\cdots9}_\text{$a_3-1$}8\overbrace{0\cdots0}^\text{$a_2-1$}1\underbrace{9\cdots9}_\text{$a_1-1$}8, \\
    81\rho&=8\underbrace{9\cdots9}_\text{$a_1-1$}1\overbrace{0\cdots0}^\text{$a_2-1$}8\underbrace{9\cdots9}_\text{$a_3-1$}1\cdots\cdots8\underbrace{9\cdots9}_\text{$a_3-1$}1\overbrace{0\cdots0}^\text{$a_2-1$}8\underbrace{9\cdots9}_\text{$a_1-1$}1,
  \end{align*}
  and we see that $r(18\rho)=81\rho\ne 18\rho$. Clearly $10\nmid 18\rho$. Now suppose that $3^\alpha\parallel \rho$ and write $\rho=3^\alpha m$. Then
  \begin{align*}
    v(18\rho) &=v(2\cdot 3^2\cdot 3^\alpha m)=v(2\cdot 3^{2+\alpha} m)=v(2\cdot 3^{2+\alpha}) +v(m), \\    
    v(81\rho)&=v(3^4\cdot3^{\alpha}m)=v(3^{4+\alpha}m)=v(3^{4+\alpha})+v(m).
  \end{align*}
  Since $v(2\cdot 3^{2+\alpha})=v(3^{4+\alpha})=7+\alpha$, we see that $v(18\rho)=v(81\rho)$. Therefore $18\rho$ is a $v$-palindrome.
\end{proof}

We now deduce two corollaries from Theorem \ref{general}, the first of which is \eqref{ten}.
\begin{corollary}\label{tenc}
  All the repeated concatenations of $18$,
  \begin{equation*}
    18, 1818, 181818,\ldots,
  \end{equation*}
  are $v$-palindromes.
\end{corollary}
\begin{proof}
  Take $\rho$ to be of the form $\rho=1010\cdots 0101$, with $0$ and $1$ alternating, in Theorem \ref{general}.
\end{proof}
\begin{corollary}\label{ac}
  All the numbers,
  \begin{equation*}
    1818, 18018, 180018,1800018,\ldots,
  \end{equation*}
  are $v$-palindromes.
\end{corollary}
\begin{proof}
  Take $\rho$ to be of the form $\rho=100\cdots 001$, with only the first and last digits being $1$ and at least one $0$, in Theorem \ref{general}.
\end{proof}

\section{Further problems}
\label{further}
In Sections \ref{sol1} and \ref{sol2}, we considered the problem of solving for $k$ in the congruence
\begin{equation}\label{prob}
  n(k)_b\equiv a\pmod{m},
\end{equation}
where $n\in\mathbb{N}$, $b\geq2$ is the base, $a\in\mathbb{Z}$, and $m\in\mathbb{N}$. This problem is inspired by the fact that all the numbers \eqref{ten} are $v$-palindromes, i.e., Corollary \ref{tenc}.

Similarly, Corollary \ref{ac} inspires another problem. Let $n[k]_b$ denote the number whose base $b$ digits are those of $n$, followed by $k$ zeros, and then another $n$ again. Then we can consider the problem of solving for $k\geq0$ in the congruence
\begin{equation*}
  n[k]_b\equiv a\pmod{m}.
\end{equation*}
Conceivably, many other variations of the problem \eqref{prob} can be considered, by restricting, in Theorem \ref{general}, the palindrome $\rho$ to a special form.

Perhaps in the most general sense, we can try to solve, for a fixed $b\geq2$, the congruence
\begin{equation*}
  n\rho\equiv a\pmod{m},
\end{equation*}
for $\rho$ a palindrome in base $b$, consisting entirely of $0$'s and $1$'s and such that between any pair of consecutive $1$'s there are at least $L-1$ zeros, where $L$ is the number of base $b$ digits of $n$. This restriction on the number of $0$'s between any pair of consecutive $1$'s is imposed so that in doing the multiplication $n\rho$, ``the copies of $n$ do not overlap''. In contrast, this restriction is not imposed in Theorem \ref{general}, and we see that in the multiplication $18\rho$, ``the copies of $18$ overlap to create $9$'s''.

\subsection*{Acknowledgements}
The author is grateful to Professor Kohji Matsumoto for comments that improved the presentation of this paper.

\bibliography{refs}{}

\begin{thebibliography}{9}

\bibitem{banks}
W.~D.~Banks, D.~N.~Hart, and M.~Sakata, Almost all palindromes are composite, \emph{Math. Res. Lett.} {\bf 11} (2004), 853--868.

\bibitem{cilleruelo}
J.~Cilleruelo, F.~Luca, and I.~E.~Shparlinski, Power values of palindromes, \emph{J. Comb. Number Theory} {\bf 1} (2009), 101--107.

\bibitem{cilleruelo13}
J.~Cilleruelo, R.~Tesoro, and F.~Luca, Palindromes in linear recurrence sequences, \emph{Monatsh. Math.} {\bf 171} (2013), 433--442.

\bibitem{harminc}
M.~Harminc and R.~Sot\'ak, Palindromic numbers in arithmetic progressions, \emph{Fibonacci Quart.} {\bf 36} (1998), 259--262.

\bibitem {luca} F.~Luca, Palindromes in Lucas sequences, \emph{Monatsh. Math.} {\bf 138} (2003), 209--223.

\bibitem {pongsriiam} P.~Pongsriiam, Longest arithmetic progressions of palindromes, \emph{J. Number Theory} {\bf 222} (2021), 362--375.

\bibitem{oeis} N. J. A. Sloane et al., The On-Line Encyclopedia of Integer Sequences, \url{https://oeis.org}, last accessed September 2021.

\bibitem {tsai} D.~Tsai, A recurring pattern in natural numbers of a certain property, \emph{Integers} {\bf 21} (2021), \#A32.

\bibitem {tsai18}  D.~Tsai, Natural numbers satisfying an unusual property, \emph{S\=ugaku Seminar} {\bf 57} (2018), 35--36 (written in Japanese).

\bibitem {tsai21}  D.~Tsai, The fundamental period of a periodic phenomenon pertaining to $v$-palindromes, preprint, 2021. Available at \url{http://arxiv.org/abs/2103.00989}.

\end{thebibliography}
\bibliographystyle{plain}

\end{document}